\documentclass[12pt]{amsart}
\usepackage{amsfonts}
\usepackage{graphicx}
\usepackage{amsmath}
\usepackage{amssymb}
\input{amssym.def}
\newtheorem{theorem}{Theorem}

\newtheorem{lemma}{Lemma}

\newtheorem{definition}{Definition}
\theoremstyle{remark}

\newcommand{\re}{\text{\rm Re }}
\newcommand{\im}{\text{\rm Im }}

\input epsf
\begin{document}

\title[Cusp]{Conformal mapping asymptotics at a cusp}
\author[D.~Prokhorov]{Dmitri Prokhorov$^1$}

\thanks{Research supported by the RF Ministry of Education and Science (project 1.1520.2014k)}

\thanks{Dmitri Prokhorov\;\;prokhorovdv@info.sgu.ru}

\thanks{$^1$ Department of Mechanics and Mathematics, Saratov State University, Astrakhanskaya Str., 83, 410012 Saratov Russia}

\maketitle

{\bf Abstract}
We describe the asymptotic behavior of the mapping function at an analytic cusp compared with
Kaiser's results for cusps with small perturbation of angles and the known explicit formulae for
cusps with circular boundary curves. We propose a boundary curve parametrization by generalized
power series which allows us to give explicit representations for locally univalent mapping
functions with given asymptotic properties and for cusp boundary curves having an arbitrary order
of tangency. \vskip3mm


\section{Introduction}

A conformal mapping from the unit disk $\mathbb D=\{z:|z|<1\}$ or the upper half-plane $\mathbb H=\{z:\im z>0\}$ onto a simply connected domain $D\subset\mathbb C$ with an analytic arc $\Gamma$ on the boundary $\partial D$ of $D$ can be extended analytically through the preimage of $\Gamma$. The situation is more complicated when $\Gamma$ consists of two arcs $\Gamma_1$ and $\Gamma_2$ meeting at a boundary point even if both $\Gamma_1$ and $\Gamma_2$ are analytic curves. Without restrictions assume that $\Gamma_1$ and $\Gamma_2$ intersect at the origin.We refer to the book \cite{Pom} by Pommerenke, Sections 3.3-3.4, for an overview. In particular, Pommerenke gives in \cite{Pom}, p.57, the asymptotic behavior of $f:\mathbb D\to f(\mathbb D)$ in the case when $\partial f(\mathbb D)$ has a Dini-smooth corner of opening $\pi\alpha$, $0<\alpha<2$.

Warschawski \cite{War} proved a geometric criterion under which a mapping function $f(z)$ behaves like $z^{\alpha}$ for a domain with opening angle $\pi\alpha$, $0<\alpha\leq2$, see also the results by Lehman \cite{Leh} about a mapping function developed in a certain generalized power series.

There are not so many works on determining the behavior of the mapping function at a cusp of a domain with a piecewise analytic boundary. We say that a domain $D\subset\mathbb C$ with $0\in\partial D$ has an {\it analytic cusp} at 0 if $\partial D$ at 0 consists of two regular analytic curves such that the interior angle of $D$ at 0 vanishes. For instance, Warschawski considered in \cite{War1} conformal mappings of infinite strips which correspond to domains with two zero interior angles, i.e., cusps, at infinity.

The explicit mapping function $f:\mathbb H\to\mathbb H\setminus\Gamma$ with the circular arc $\Gamma$ of radius 1 and centered at $i$ is shown in \cite{Pro}. The curves $\Gamma$ and $R^+:=\{x\in\mathbb R:x\geq0\}$ form the analytic cusp at 0. Note that $\Gamma$ has the first order tangency to $\mathbb R=\partial\mathbb H$ at 0 and the curvature of $\Gamma$ at 0 equals 1. We say that two curves in $\mathbb R^2$ defined by $y=g(x)$ and $y=h(x)$ have the $n$-th order tangency at the point $P=(x_0,y_0)$ if $g(x_0)=h(x_0)=y_0$, $g'(x_0)=h'(x_0),\dots,g^{(n)}(x_0)=h^{(n)}(x_0)$ and $g^{(n+1)}(x_0)\neq h^{(n+1)}(x_0)$. It is possible to generalize this concept from $n\in\mathbb N$ to an arbitrary $p$-order tangency, $p>0$.

Kaiser investigated in \cite{Kai} analytic cusps possessing a perturbation property. For a domain $D$, denote the boundary curves of an analytic cusp at $0\in\partial D$ by $\Gamma_1$ and $\Gamma_2$, $\Gamma_k=\Gamma_k[0,\epsilon]=te^{i\sphericalangle_k(t)}$ where $\sphericalangle_k(t)$ are real power series convergent in a neighborhood of $t=0$, $\sphericalangle_k(0)=0$, $k=1,2$, $\epsilon>0$, $\sphericalangle_D(t):=\sphericalangle_2(t)-\sphericalangle_1(t)$. For $h(t)=\sum_{n=1}^{\infty}a_nt^n$, $\text{ord}(h):=\min\{n\in\mathbb N:a_n\neq0\}$. Let $d:=\text{ord}(\sphericalangle_D)$ and $a>0$ be such that $\lim_{t\to0}\sphericalangle_D(t)/at^d=1$. We say that $D$ has small perturbations of angles if $\min\{\text{ord}(\sphericalangle_1),\text{ord}(\sphericalangle_2)\}=d$ and $\text{ord}(\sphericalangle_D(t)-at^d)>2d$. Kaiser proved in \cite{Kai} that, for $D\subset\mathbb C$ with an analytic cusp $0\in\partial D$ and having small perturbations of angles, the conformal map $f:\mathbb H\to D$, $f(0)=0$, satisfies the asymptotic relation
\begin{equation}
\lim_{z\to0}f(z)\left(-\frac{\pi}{da\log|z|}\right)^{-\frac{1}{d}}=1,\;\;\im z>0. \label{kai}
\end{equation}

In this article, instead of Kaiser's parametrization $\Gamma_k=te^{i\sphericalangle_k(t)}$, $k=1,2$, we propose
\begin{equation}
\Gamma_k(z)=\frac{-2\pi}{a\log z}\left(1+\sum_{n-1}^{\infty}\frac{\Phi_n(z)}{\log^nz}\right),\;\;k=1,2, \label{pro}
\end{equation}
where $z\in(-\epsilon,0)$ for $\Gamma_1$ and $z\in(0,\epsilon)$ for $\Gamma_2\subset\mathbb R$, and the continuous branch of $\log z$ for $z$ from the closure of $\mathbb H$ is defined by $\log i=i\frac{\pi}{2}.$ Parametrization (\ref{pro}) for $\Gamma_k$ is motivated by the Christoffel-Schwarz representation in \cite{Pro} for a circular arc $\Gamma$. The functions $\Phi_n$ are developed in certain power series with real coefficients.

The article is organized as follows. Section 2 contains preliminary results preparing the proof of the main Theorem 1. In Section 3 we prove Theorem 1 showing that (\ref{pro}) implies the conformal mapping asymptotics in a cusp formed by the interval $\Gamma_2$ and the curve $\Gamma_1=\Gamma$ having the first order tangency to $\mathbb R$ at 0. It is interesting to compare the new results with Kaiser's formula (\ref{kai}) where $d=1$. Theorem 2 generalizes the asymptotic relation (\ref{kai}) for all $d>0$ applied to the curve $\Gamma^{\frac{1}{d}}$ which has the $d-th$ order tangency to $\mathbb R$ at 0.

\section{Preliminaries}

Consider a circular arc $\gamma=\gamma(-\frac{\pi}{2},\varphi_0]:=\{i+e^{i\varphi}:-\frac{\pi}{2}<\varphi\leq\varphi_0\}$, $-\frac{\pi}{2}<\varphi_0<0$, in the upper half-plane $\mathbb H$, $\gamma(0)=0$, and a mapping $w=f(z)$ from $\mathbb H$ onto $\mathbb H\setminus\gamma$. The mapping $f$ with the hydrodynamic normalization at infinity has been constructed in \cite{Pro} applying the Christoffel-Schwarz formula. Set $f(0)=0$ and rewrite $f(z)$ from \cite{Pro} in the form
\begin{equation}
\frac{1}{f(z)}=\frac{1}{2\pi}\left[\log\frac{z+\alpha}{z}+\frac{2\pi-\alpha}{z+\alpha}\right], \label{Chr}
\end{equation}
where $\alpha>0$ depends on $\varphi_0$.

On the boundary of $\mathbb H\setminus\gamma$ we see the corner of opening $\pi$ at $w=0$ and the cusp with the mutual boundary arc $\gamma$. Representation (\ref{Chr}) allows us to expand $f(z)$ near $z=0$, $z\in\mathbb H$,
\begin{equation}
f(z)=\frac{2\pi}{-\log z+\log(z+\alpha)+(2\pi-\alpha)/(z+\alpha)}= \label{exp}
\end{equation}
$$\frac{-2\pi}{\log z}\left[1+\sum_{n=1}^{\infty}\left( \frac{\log(z+\alpha)+(2\pi-\alpha)/(z+\alpha)}{\log z}\right)^n\right]=\frac{-2\pi}{\log z}\left(1+\sum_{n=1}^{\infty}\sum_{k=0}^{\infty}\frac{c_{nk}z^k}{\log^nz}\right)$$ with certain real coefficients $c_{nk}$, $n\geq1$, $k\geq0$, for which $\sum_{k=0}^{\infty}c_{nk}x^k$ converges for all $n\geq1$ and real $x$ small enough and $$\sum_{n=1}^{\infty}\sum_{k=0}^{\infty}\frac{c_{nk}x^k}{\log^nx}$$ converges in a neighborhood of $x=0$, $x\in\mathbb R$.

Extend $f(z)$ continuously onto $\mathbb H\cup\mathbb R$ and parameterize $\gamma$ in terms of parameter $x<0$ near $x=0$,
$$
f(x)=\frac{-2\pi(\log|x|-i\pi)}{\log^2|x|+\pi^2}\left[1+\sum_{n=1}^{\infty}\left(\frac{\log(x+\alpha)+(2\pi-\alpha)/(x+\alpha)} {\log|x|+i\pi}\right)^n\right],
$$
$$f(x)=u(x)+iv(x),$$ $$u(x)=\frac{-2\pi}{\log|x|}+O\left(\frac{1}{\log^2|x|}\right),\;\;x\to0,\;\;x<0,$$ $$v(x)=\frac{2\pi^2}{\log^2|x|}+O\left(\frac{1}{\log^3|x|}\right),\;\;x\to0,\;\;x<0.$$ This implies the explicit asymptotic representation for the circular arc $\gamma$ of curvature 1 at $u=0$, $$v=\frac{u^2}{2}+o(u^2),\;\;u\to0,\;\;u>0.$$

Generalize representation (\ref{exp}) with real coefficients $\{c_{nk}\}$ determined by the circular arc $\gamma$ of curvature 1 to the same formula (\ref{exp}) with rather free coefficients.

Let a curve $\Gamma=\Gamma[-x_0,0)$ be parameterized by the parameter $x\in[-x_0,0)$, $x_0>0$, as follows $$\Gamma[-x_0,0)=\{u(x)+iv(x): x\in[-x_0,0)\},$$
$$
u(x)=\re\left[\frac{-2\pi(\log|x|-i\pi)}{a(\log^2|x|+\pi^2)}\left(1+\sum_{n=1}^{\infty}\sum_{k=0}^{\infty}\frac{c_{nk}x^k}{\log^nx}\right) \right],
$$
$$
v(x)=\im\left[\frac{-2\pi(\log|x|-i\pi)}{a(\log^2|x|+\pi^2)}\left(1+\sum_{n=1}^{\infty}\sum_{k=0}^{\infty}\frac{c_{nk}x^k}{\log^nx}\right) \right],
$$
where $a>0$ is the curvature of $\Gamma$ at $x=0$.

\begin{definition}
We say that the representation
\begin{equation}
f(z)=\frac{-2\pi}{a\log z}\left(1+\sum_{n=1}^{\infty}\sum_{k=0}^{\infty}\frac{c_{nk}z^k}{\log^nz}\right) \label{rep}
\end{equation}
with $a>0$ and real coefficients $c_{nk}$, $n\geq1$, $k\geq0$, is admissible if the following conditions are satisfied: \newline \text{(i)} the set $\{\root k \of{|c_{nk}|}:n\geq1,k\geq0\}$ is bounded; \newline \text{(ii)} the sets $\{\root n\of{|\sum_{k=0}^{\infty}c_{nk}z^k}|:\im z\geq0\}$, $n\geq1$, are uniformly bounded for $|z|$ small enough; \newline \text{(iii)} the sets $\{\root n\of{|\sum_{k=1}^{\infty}kc_{nk}z^{k-1}}|:\im z\geq0\}$, $n\geq1$, are uniformly bounded for $|z|$ small enough; \newline \text{(iv) } $c_{10}\neq0$.
\end{definition}

For $z$ from a neighborhood of the origin and for an admissible function (\ref{rep}), denote $$\Phi_n(z)=\sum_{k=0}^{\infty}c_{nk}z^k,\;\;n\geq1,\;\;\Phi_1(0)\ne0.$$ The functions $\Phi_n$ and $$f(z)=\frac{-2\pi}{a\log z}\left(1+\sum_{n=1}^{\infty}\frac{\Phi_n(z)}{\log^nz}\right)$$ are analytic in $H_R\setminus\{0\}$ where $H_R:=\{z:|z|<R,\im z\geq0\}$, $R>0$ is small enough. We will show that $f(x)$, $x\in\mathbb R$, parameterizes an interval $(0,R']\subset\mathbb R$ for $x>0$ in a neighborhood of $x=0$, and a Jordan curve $\Gamma\subset\mathbb H$ for $x<0$ in a neighborhood of $x=0$.

\begin{lemma} An admissible representation $f(x)$ is a one-to-one parametrization of an interval $(0,R')\subset\mathbb R$ for $x>0$ small enough.
\end{lemma}

\begin{proof} We have to show that $$f(x)=\frac{-2\pi}{a\log x}\left(1+\sum_{n=1}^{\infty}\frac{\Phi_n(x)}{\log^nx}\right),\;\;x>0,$$ has a positive derivative $f'(x)$. Indeed, $$f'(x)=\frac{2\pi}{a\log^2x}\left[\frac{1}{x}+\sum_{n=1}^{\infty}\frac{(n+1)\Phi_n(x)}{x\log^nx}- \sum_{n=1}^{\infty}\frac{\Phi'_n(x)}{\log^{n-1}x}\right].$$ According to conditions (ii) and (iii) of Definition 1, both series in the latter formula converge for $x>0$ small enough and $$\sum_{n=1}^{\infty}\frac{(n+1)\Phi_n(x)}{\log^nx}=O\left(\frac{1}{\log x}\right),\;\; \sum_{n=1}^{\infty}\frac{\Phi'_n(x)}{\log^{n-1}x}=O(1),\;\;x\to+0.$$ Hence, $$f'(x)=\frac{2\pi}{a\log^2x}\left[\frac{1}{x}\left(1+O\left(\frac{1}{\log x}\right)\right)+O(1)\right],\;\;x\to+0,$$ which implies that $f'(x)>0$ for $x>0$ small enough and completes the proof of Lemma 1.
\end{proof}

\begin{lemma} An admissible representation $f(x)$ is a one-to-one parametrization of a simple curve $\Gamma\subset\mathbb H$ for $x<0$ small enough.
\end{lemma}

\begin{proof} It is sufficient to show that $$f(x)=\frac{-2\pi(\log|x|-i\pi)}{a(\log^2|x|+\pi^2)}\left[1+\sum_{n=1}^{\infty}\frac{\Phi_n(x)(\log|x|-i\pi)^n} {(\log^2|x|+\pi^2)^n}\right]=u(x)+iv(x)$$ satisfies the conditions $v(x)>0$ and $u'(x)<0$ for $x<0$ small enough.

The functions $\root n \of{|\Phi_n(x)|}$ are uniformly bounded in a neighborhood of $x=0$. Therefore $$v(x)=\frac{2\pi}{a(\log^2|x|+\pi^2)}\left[\pi- \frac{\Phi_1(x)\im(\log|x|-i\pi)^2}{\log^2|x|+\pi^2}-\right.$$ $$\left.\sum_{n=2}^{\infty}\frac{\Phi_n(x)\im(\log|x|-i\pi)^{n+1}} {(\log^2|x|+\pi^2)^n}\right]=$$ $$\frac{2\pi}{a(\log^2|x|+\pi^2)}\left[\pi+\frac{\Phi_1(x)2\pi\log|x|}{\log^2|x|+\pi^2}+ O\left(\frac{1}{\log|x|}\right)\right],\;\;x\to-0,$$ which implies that $v(x)>0$ in a left neighborhood of $x=0$.

Calculate $u'(x)$, $$u'(x)=\frac{2\pi}{ax}\left[\frac{\log^2|x|-\pi^2}{(\log^2|x|+\pi^2)^2}+\right.$$ $$\left.\sum_{n=1}^{\infty} \frac{(n+1)\Phi_n(x)\re((\log|x|-i\pi)^n(\log^2|x|-\pi^2-i2\pi\log|x|))}{(\log^2|x|+\pi^2)^{n+2}}\right]-$$ $$\frac{2\pi}{a}\sum_{n=1}^{\infty} \frac{\Phi'_n(x)\re(\log|x|-i\pi)^{n+1}}{(\log^2|x|+\pi^2)^{n+1}}=$$ $$\frac{2\pi}{a\log^2|x|}\left[\frac{1}{x}\left(1+O\left(\frac{1}{\log|x|}\right)\right)+O(1)\right],\;\;x\to-0,$$ which implies that $u'(x)<0$ in a left neighborhood of $x=0$ and completes the proof of Lemma 2.
\end{proof}

A curve $\Gamma$ defined by an admissible representation (\ref{rep}) has the first order tangency to $\mathbb R$ and curvature $a$ at the origin.

\section{Main results}

Now we are able to present the result on the asymptotic behavior of a conformal mapping in a neighborhood of a cusp with boundary arcs $(0,x_0)$ and $\Gamma$ given by an admissible representation (\ref{rep}).

\begin{theorem}
Let $D\subset\mathbb C$ be a simply connected domain having an analytic cusp at 0 with the boundary analytic arcs $[0,\epsilon]\subset\mathbb R$, $\epsilon>0$, and $\Gamma[-\epsilon,0]$ given by an admissible representation (\ref{rep}). Let $g:\mathbb H\to D$ be a conformal map, $g(0)=0$. Then
\begin{equation}
\lim_{z\to0}g(z)\left(-\frac{2\pi}{a\log|z|}\right)^{-1}=1,\;\;\im z>0. \label{th1}
\end{equation}
\end{theorem}

\begin{proof}
According to Lemma 1, the admissible function $f(x)$ from (\ref{rep}) maps an interval $(0,\epsilon_0)$ one-to-one onto an interval $(0,R')\subset\mathbb R$. Similarly, this function maps an interval $(-\epsilon_0,0)$ one-to-one onto a simple curve $\Gamma$. Together with $(0,R')$, $\Gamma$ forms the analytic cusp at 0.

For $\epsilon\in(0,\epsilon_0)$, denote by $C_{\epsilon}$ the half-circle, $$C_{\epsilon}:=\{z:|z|=\epsilon,\im z\geq0\},$$ and by $\mathbb D^+_{\epsilon}$ the half-disk bounded by $[-\epsilon,\epsilon]$ and $C_{\epsilon}$. Suppose that $\epsilon$ is so small that $f(\mathbb D^+_{\epsilon})\subset D$. We wish to show that $f(z)$ is univalent in $\mathbb D^+_{\epsilon}$. To this purpose, it is sufficient to prove that $f(C_{\epsilon})$ is a simple curve $L$ connecting $f(\epsilon)\in(0,R')$ and $f(-\epsilon)\in\Gamma$.

Let $$f(\epsilon e^{i\varphi})=u_{\epsilon}(\varphi)+iv_{\epsilon}(\varphi),\;\;0\leq\varphi\leq\pi,$$ so that $$v_{\epsilon}(\varphi)=\frac{2\pi}{a}\left[\frac{\varphi}{\log^2\epsilon+\varphi^2}-\sum_{n=1}^{\infty}\frac{\im(\Phi_n(e^{i\varphi}) (\log\epsilon-i\varphi)^{n+1})}{(\log^2\epsilon+\varphi^2)^{n+1}}\right],\;\;0\leq\varphi\leq\pi.$$

As far as $\Gamma$ behaves like a parabola of curvature $a$ at 0, the curve $L$ is simple and meets $\Gamma$ only at $f(-\epsilon)$ provided $v'_{\epsilon}(\varphi)>0$ and $u'_{\epsilon}(\varphi)<0$. Note that $\root n\of{|\Phi_n(\epsilon e^{i\varphi})|}$ and $\root n\of{|\Phi'_n(\epsilon e^{i\varphi})|}$, $n\geq1$, are uniformly bounded on $[0,\pi]$ and calculate $v'_{\epsilon}(\varphi)$, $$v'_{\epsilon}(\varphi)=\frac{2\pi}{a}\left[\frac{\log^2\epsilon-\varphi^2}{(\log^2\epsilon+\varphi^2)^2}+\right.$$  $$\left.\sum_{n=1}^{\infty}\frac{(n+1)\im(\Phi_n(\epsilon e^{i\varphi})(\log\epsilon-i\varphi)^n(2\varphi\log\epsilon+i(\log^2\epsilon-\varphi^2)))}{(\log^2\epsilon+\varphi^2)^{n+2}}-\right.$$ $$\left.\sum_{n=1}^{\infty}\frac{\im(\Phi'_n(\epsilon e^{i\varphi})i\epsilon e^{i\varphi}(\log\epsilon-i\varphi)^{n+1})} {(\log^2\epsilon+\varphi^2)^{n+1}}\right].$$

Hence $$v'_{\epsilon}(\varphi)=\frac{2\pi}{a(\log^2\epsilon+\varphi^2)}\left[1+O\left(\frac{1}{\log\epsilon}\right)\right]>0,\;\;\epsilon\to0.$$

We have also $$u_{\epsilon}(\varphi)=\frac{-2\pi}{a}\left[\frac{\log\epsilon}{\log^2\epsilon+\varphi^2}+\sum_{n=1}^{\infty}\frac{\re(\Phi_n(e^{i\varphi}) (\log\epsilon-i\varphi)^{n+1})}{(\log^2\epsilon+\varphi^2)^{n+1}}\right],\;\;0\leq\varphi\leq\pi.$$

Calculate $u'_{\epsilon}(\varphi)$,
\begin{equation} u'_{\epsilon}(\varphi)=\frac{2\pi}{a(\log^2\epsilon+\varphi^2)}\left[\frac{2\varphi\log\epsilon} {\log^2\epsilon+\varphi^2}+\right. \label{der}
\end{equation}
$$\sum_{n=1}^{\infty}\frac{(n+1)\re(\Phi_n(\epsilon e^{i\varphi})(\log\epsilon-i\varphi)^n(2\varphi\log\epsilon+ i(\log^2\epsilon-\varphi^2)))}{(\log^2\epsilon+\varphi)^{n+1}}-$$ $$\left.\sum_{n=1}^{\infty}\frac{\re(\Phi'_n(\epsilon e^{i\varphi})\epsilon ie^{i\varphi}(\log\epsilon-i\varphi)^{n+1})} {(\log^2\epsilon+\varphi^2)^n}\right].$$

Find the asymptotical expansion in powers of $\varphi$ for every term in (\ref{der}) with the main coefficient depending on $\epsilon$, $$\frac{2\varphi\log\epsilon}{\log^2\epsilon+\varphi^2}=\frac{2\varphi}{\log\epsilon}\left(1+O\left(\frac{1}{\log\epsilon}\right)\right)+ O(\varphi^2),\;\;\varphi\to0,\;\;\epsilon\to0,$$ $$\sum_{n=1}^{\infty}\frac{(n+1)\re(\Phi_n(\epsilon e^{i\varphi})(\log\epsilon-i\varphi)^n(2\varphi\log\epsilon+ i(\log^2\epsilon-\varphi^2)))}{(\log^2\epsilon+\varphi)^{n+1}}=$$ $$\frac{\varphi}{\log\epsilon}O\left(\frac{1}{\log\epsilon}\right)+O(\varphi^2),\;\;\varphi\to0,\;\;\epsilon\to0,$$ $$\sum_{n=1}^{\infty}\frac{\re(\Phi'_n(\epsilon e^{i\varphi})\epsilon ie^{i\varphi}(\log\epsilon-i\varphi)^{n+1})} {(\log^2\epsilon+\varphi^2)^n}=\varphi\;O(\epsilon)+O(\varphi^2),\;\;\varphi\to0,\;\;\epsilon\to0.$$

Consequently, $$u'_{\epsilon}(\varphi)=\frac{2\pi}{\log^3\epsilon}2\varphi\left(1+O\left(\frac{1}{\log\epsilon}\right)\right)+O(\varphi^2)<0$$ for $0<\varphi\leq\varphi_0$ where $\varphi_0>0$ and $\epsilon>0$ are small enough.

Now let us examine $u'_{\epsilon}(\varphi)$ on $[\varphi_0,\pi]$. Take into account condition (iv) of Definition 1 and observe that (\ref{der}) implies the following inequalities  $$u'_{\epsilon}(\varphi)\leq\frac{2\pi}{a(\log^2\epsilon+\varphi^2)}\left[\frac{2\varphi_0\log\epsilon} {\log^2\epsilon+\varphi^2}+\right.$$ $$\max_{\varphi_0\leq\varphi\leq\pi}\sum_{n=1}^{\infty}\frac{(n+1)\re(\Phi_n(\epsilon e^{i\varphi})(\log\epsilon-i\varphi)^n(2\varphi\log\epsilon+ i(\log^2\epsilon-\varphi^2)))}{(\log^2\epsilon+\varphi)^{n+1}}+$$ $$\left.\max_{\varphi_0\leq\varphi\leq\pi}\sum_{n=1}^{\infty}\frac{\im(\Phi'_n(\epsilon e^{i\varphi})\epsilon e^{i\varphi}(\log\epsilon-i\varphi)^{n+1})} {(\log^2\epsilon+\varphi^2)^n}\right]=$$
$$\frac{2\pi}{a(\log^2\epsilon+\varphi^2)^2}[2\varphi_0\log\epsilon+o(\log\epsilon)+O(\epsilon\log^2\epsilon)]<0$$ for $\epsilon>0$ small enough.

Thus $f(z)$ is locally univalent at 0, $f$ maps the half-disk $\mathbb D_{\epsilon}^+$ onto $f(\mathbb D_{\epsilon^+})$ and is continuous on $\mathbb D_{\epsilon}^+\cup[-\epsilon,\epsilon]$. By (\ref{rep}), $f(z)$ possesses the asymptotical property at 0,
\begin{equation}
\lim_{z\to0}f(z)\left(-\frac{2\pi}{a\log|z|}\right)^{-1}=1,\;\;\im z>0, \label{loc}
\end{equation}
which is invariant under an analytic change of variables $z=z(\zeta)$ preserving 0.

Let $g$, $g(0)=0$, be a conformal map from $\mathbb H$ onto a simply connected domain $D$ having an analytic cusp at 0 with the boundary analytic arcs $[0,\epsilon]\subset\mathbb R$, $\epsilon>0$, and $\Gamma[-\epsilon,0]$ given by an admissible representation (\ref{rep}). Denote by $h:\mathbb D_{\epsilon}^+\to f^{-1}(g(\mathbb D_{\epsilon}^+))$ a conformal map preserving 0. Then $h^{-1}\circ f^{-1}\circ g$ is an automorphism $\Psi$ of $\mathbb D_{\epsilon}^+$ preserving 0. Hence $$g=f\circ h\circ\Psi.$$ The functions $\Psi$ and $h$ are analytic at 0 because both of them can be analytically extended through an interval on $\mathbb R$ containing 0. Therefore, it follows from (\ref{loc}) that $g$ and $f$ possess the same asymptotical property (\ref{th1}) at 0 which completes the proof of Theorem 1.
\end{proof}

Kaiser \cite{Kai} proved the asymptotic relation (\ref{kai}) for natural $d\in\mathbb N$. Theorem 1 gives a partial Kaiser's relation (\ref{kai}) for $d=1$ and $a$ substituted by $2a$ with another parametrization of the analytic cusp boundaries. However it is possible to generalize Theorem 1 and Kaiser's results to all $d>0$.

\begin{theorem}
For any $d>0$, let $D^{\frac{1}{d}}\subset\mathbb C$ be a simply connected domain having an analytic cusp at 0 with the boundary analytic arcs $[0,\epsilon]\subset\mathbb R$, $\epsilon>0$, and $\Gamma^{\frac{1}{d}}[-\epsilon,0]$ where $\Gamma$ is given by an admissible representation (\ref{rep}). Let $g:\mathbb H\to D^{\frac{1}{d}}$ be a conformal map, $g(0)=0$. Then
\begin{equation}
\lim_{z\to0}g(z)\left(-\frac{2\pi}{da\log|z|}\right)^{-\frac{1}{d}}=1,\;\;\im z>0. \label{th2}
\end{equation}
The branch of $w^{\frac{1}{d}}$ here and in the sequel is taken so that $1^{\frac{1}{d}}=1$.
\end{theorem}

\begin{proof}
As in Theorem 1, reduce the proof to a domain $D$ in a neighborhood of 0. For a given $d>0$, choose a domain $D$ in a neighborhood of the cusp at 0 such that the domains $D$ and $D^{\frac{1}{d}}$ are domains in $\mathbb C$ with the cusp at 0 with the boundary arcs $[0,\epsilon]\subset\mathbb R$, $\epsilon>0$, and $\Gamma[-\epsilon,0]$ and $\Gamma^{\frac{1}{d}}[-\epsilon,0]$, respectively, where $\Gamma$ is given by an admissible representation (\ref{rep}).

Apply Theorem 1 and find a function $f:\mathbb H\to D$ for which relation (\ref{loc}) is valid. Then the function $$g(z)=f^{\frac{1}{d}}(z^d)$$ satisfies the conditions of Theorem 2. The relation $$\lim_{z\to0}f^{\frac{1}{d}}(z^d)\left(-\frac{2\pi}{a\log z^d}\right)^{-\frac{1}{d}}=1$$ implies (\ref{th2}) and completes the proof of Theorem 2.
\end{proof}

Note that the real part $u(z)$ and the imaginary part $v(z)$ of the admissible function $f(z)=u(z)+iv(z)$ in (\ref{rep}) have the asymptotic expansion $$v(z)=\frac{a}{2}u^2(z)+O(u^3(z)),\;\;u\to0.$$ In Theorem 2 we deal with $g(z)=f^{\frac{1}{d}}(z^d)=u^*(z)+iv^*(z)$ where $u^*(z)$ and $v^*(z)$ have the asymptotic expansion $$v^*(z)=\frac{a}{2d}(u^*(z))^{1+d}+o(u^*(z)^{1+d}),\;\;u^*\to0.$$ This means that the curve $\Gamma^{\frac{1}{d}}$ in Theorem 2 has the $d$-th order tangency to $\mathbb R$ at 0. The coefficient $\frac{a}{d}$ does not characterize the curvature of $\Gamma^{\frac{1}{d}}$ at 0 if $d\neq1$ but it can be useful for describing geometrical properties of $\Gamma^{\frac{1}{d}}$ at 0.

\end{document}